\documentclass[12pt,letterpaper, english]{article} 
\usepackage[activeacute,english]{babel}
\usepackage[]{inputenc}%acentos
\usepackage{moreverb}
\usepackage{natbib}
\usepackage{amsmath}
\usepackage{graphicx,color} 
\usepackage{amsfonts} 
\usepackage{amssymb}
\usepackage{amsthm}
\usepackage{amsxtra}
\usepackage{amsgen}
\usepackage{amscd}
\usepackage{tensor}
\usepackage{multirow}
\usepackage{latexsym}
 \usepackage{tabulary} 
\usepackage [colorlinks, bookmarks open, bookmarks numbered,citecolor= blue, urlcolor= blue]{hyperref} % if run by PDFLaTex 
\usepackage{float}
\newtheorem{thm}{Theorem}[section]

\newtheorem{defn}{Definition}[section]

\date{} 
\begin{document}
\title{Dynamics of four families of methods with the same weight function to solve nonlinear equations.}
\maketitle  
\begin{center}
\author{\textbf{ Livia J. Qui\~nonez T.$^{1}$ and Carlos E. Cadenas R.$^{1,2}$}}\\           
{\scriptsize
 $^{1}$ Department of Mathematics, Experimental Faculty of Science and Technology, University of Carabobo, Venezuela.\\
 $^{2}$ Multidisciplinary Center for Visualization and Scientific Computing, University of Carabobo, Venezuela .\\
}
\end{center}	
\begin{abstract}
We study the dynamics of four families of methods obtained with a weight function from a convex combination of Newton's method and a Newton-Halley type method on polynomials with two roots. We find the analytical expressions for the fixed and critical points. We study the stable and unstable behavior of the strange fixed points. Also, parameters spaces for identify methods with good behavior are presented. Then, several dynamic planes are presented to confirm the results obtained. Finally, some  periodic orbits with period two for a selected method are presented.\\
\end{abstract}

\footnote{ Livia J. Qui\~nonez T., liviaq33@hotmail.com }
\footnote{Corresponding author: Carlos E. Cadenas R., ccadenas@uc.edu.ve , ccadenas45@gmail.com}

%%%%%%%%%%%%%%%%%%%%%%%%%%%%%%%%%%%%%%%%%%%%%%%%%%%%%%%%%%%%%%%%%%%%%%%%%%%%%%%%%%%%%%%%%%%%%%%%%%%%%%%%%%%%%%%%%%%%%%%%%%%%%%%%%%%%

\section{Introduction}\label{intro}
%%% add text here for the Introduction
Iterative methods are necessary usually for solving nonlinear equations $f(x)=0$ (\cite{JTraub1964}-\cite{IKArgyros1993}). Several good methods exist in the literature: Newton, Halley and Chebyshev methods between others.
In previous papers the study of the dynamic to various methods is done. Initially the dynamics of the methods and families of a point of order two and three were presented. Then, this approach is used to study the behavior of families of one point and multipoints methods. See (\cite{PBlanchard1984}-\cite{CECadenas2017a1}) for example.
%\section{Family based in a convex combination of the Newton's method and one Newton-Halley type method}
%\label{}
In \cite{CECadenas2017a00} the weight function
\begin{equation}\label{eq:H}
	H(t)=A+\frac{2(A-1)^2}{2(1-A)-t}\nonumber
\end{equation}

is used to obtain the following family of one point given by:
\begin{equation}\label{family}
	x_{n+1}=x_n-\frac{f(x_n)}{f'(x_n)}\left(A+\frac{2(A-1)^2}{2(1-A)-L_f(x_n)}\right);\quad n=0,1,2,\cdots\quad
\end{equation}
where $L_{f}(x)=\frac{f(x)f''(x)}{(f'(x))^2}.$

If $A\neq 1$ this method have third order of convergence and its error equation is given by
\begin{equation}
	e_{n+1}=\left(\frac{(1-2A)}{4(1-A)}\left[\frac{f''(\alpha)}{f'(\alpha)}\right]^2-\frac{1}{6}\frac{f'''(\alpha)}{f'(\alpha)}\right)e_n^3+O(e_n^4)
\end{equation}

The weight function $H$ given in (1) is used in \cite{CECadenas2018c} to obtain a convex combination of the Newton's method and one Newton-Halley type method to multiple roots
\begin{equation} \label{familym}
	x_{n+1}=x_n-m\frac{f(x_n)}{f'(x_n)}\left(A+\frac{2(A-1)^2}{2(1-A)-(1-m+mL_f(x_n))}\right); (A\neq1)
\end{equation}
with third order of convergence and asymptotic error constant  $$K(\alpha)=\frac{1}{2m^2}\left[\left(\frac{(m+3)(1-A)-2}{1-A}\right)A_1^2-2mA_2\right].$$

If we use the weight function (\ref{eq:H}) in \cite{CECadenas2019} we get the following family of two points:
\begin{eqnarray}
	y_{n}&=& x_n-\frac{f(x_n)}{f'(x_n)}\\
	x_{n+1}&=&x_n-\frac{f(x_n)}{f'(x_n)}\left(A+\frac{(A-1)^2f(x_n)}{(1-A)f(x_n)-f(y_{n})}\right)\label{family2}
\end{eqnarray}

If we use the weight function (\ref{eq:H}) in \cite{CECadenas2019b} we get the following family of two points:
\begin{eqnarray}\label{Eq3_1}
	y_n &=& x_n-\frac{2}{3}\frac{f(x_n)}{f'(x_n)} \\
	x_{n+1} &=& x_n-\frac{f(x_n)}{f'(x_n)}\left(A+\frac{2(A-1)^2}{2(1-A)-\frac{f'(x_n)-f'(y_n)}{\frac{2}{3} f'(x_n)}}\right)\label{GanderJarrattClass1}
\end{eqnarray}

On the other hand, it is well known that the families of one and two points given by (\ref{family}), (\ref{familym}), (\ref{family2}) and (\ref{GanderJarrattClass1}) have the same iteration equation when applied to polynomials with two different roots. In the case of the family given by (\ref{familym}), this is true if the two roots have multiplicity $m$ (see \cite{CECadenas2019b}).

In this work, the dynamics of the families (\ref{family}), (\ref{familym}), (\ref{family2}) and (\ref{GanderJarrattClass1}) for the class of polynomials with two different roots is studied simultaneously.

Now, we are going to recall some preliminaries of complex dynamics (see \cite{PBlanchard1984}, \cite{SAmat2004} and \cite{JMilnor2006}) that we use in this work.

\subsection{\textbf{Basic preliminaries}}

Given a rational function $R:\widehat{\mathbb{C}}\rightarrow\widehat{\mathbb{C}}$, where $\widehat{\mathbb{C}}$ is the Riemann sphere

\begin{defn} For $z\in\widehat{\mathbb{C}}$ we define its orbit as the set $$orb(z)=\left\{z,R(z),R^2(z),\cdots,R^n(z),\cdots\right\}.$$
\end{defn}

\begin{defn} A periodic point $z_0$ of period $m>1$ is a point such that $R^m(z_0)=z_0$ and $R^k(z_0)\neq z_0$, for $k<m$.
\end{defn}

\begin{defn} A pre-periodic point is a point $z_0$ that is not periodic but exist a $k>0$ such that $R^k$ is periodic.
\end{defn}

\begin{defn} A point $z_0$ is a fixed point of $R$ if $R(z_0)=z_0$.
\end{defn}

\begin{defn} A critical point $z_{cr}$ is a point such that $R'(z_{cr})=0$.
\end{defn}

\begin{defn} A fixed point $z_0$ is called attractor if $|R'(z_0)|<1$, repulsive if $|R'(z_0)|>1$, and parabolic or neutral if $|R'(z_0)|=1$. If $|R'(z_0)|=0$ then the fixed point is called superattractor. A fixed point superattractor is also a critical point.
\end{defn}

\begin{defn} A fixed point $z_0$ that is not associated to the roots of the function $f(z)$ is called strange fixed point.
\end{defn}

\begin{defn}The basin of attraction of a attractor $\alpha\in\widehat{\mathbb{C}}$ is defined as the set of starting points whose orbits tend to $\alpha$.
\end{defn}

%%%%%%%%%%%%%%%%%%%%%%%%%%%%%%%%%%%%%%%%%%%%%%%%%%%%%%%%%%%%%%%%%%%%%%%%%%%%%%%%%%%%%%%%%%%%%%%%%%%%%%%%%%%%%%%%%%%%%%%%%%%%%%%%%%%%

\section{Dynamical behavior of the rational function associated with the family in study}

Here the conjugacy class is established and the analytical expressions for the fixed and critical points  of this family in terms of the parameter $A$ is obtained. Then study of the fixed points, critical points and parameter space are presented. To finish this section several dynamical planes to different values of $A$ selected from parameter space are shown.

%%% ----------------------------------------------------------------------

\subsection{Conjugacy classes}

Here the authors establish the conjugacy class and the analytical expressions for the fixed and critical points of this family in terms of the parameter $A$ is obtained.

Throughout the remainder of this paper we study the dynamics of the rational map $R_f$ arising from the method (\ref{family})

\begin{equation}\label{familyR}
	R_f=z-\frac{f(z)}{f'(z)}\left(A+\frac{2(A-1)^2}{\left(2(1-A)-L_f(z)\right)}\right); \textrm{ where }  L_{f}(z)=\frac{f(z)f''(z)}{\left(f'(z)\right)^2}
\end{equation}
applied to a generic polynomial of two grade with simple roots. As mentioned above, the dynamic study for $R_f$ is the same as for the rational maps from the iteration equations  (\ref{familym}), (\ref{family2}) and (\ref{GanderJarrattClass1}) when applied to polynomials with two different roots. Let us first recall the definition of analytic conjugacy classes.

\begin{defn} \cite{AFBeardon1991}. Let $f$ and $g$ be two maps from the Riemann sphere into itself. An analytic conjugacy between  $f$ and $g$ is an analytic diffeomorphism $h$ from the Riemann sphere onto itself such that $h\circ f=g\circ h$.
\end{defn}

$R_f$ and $R_g$ have the following property for the analytic functions $f$ and $g$:

\begin{thm}\label{t:scaling} (The Scaling Theorem). Let $f(z)$ be an analytical function on the Riemann sphere, and let $T(z)=\alpha z+\beta$, $\alpha\neq0$, be an affine map. If $g(z)=f\circ T(z)$, then $T\circ R_g\circ T^{-1}=R_f(z)$. That is, $R_f$ is analytically conjugate to $R_g$ by $T$.
\end{thm}

\begin{proof}
	With the iteration functions $R_f$ and $R_g$ given in (\ref{familyR}), we have
	
	$$R_g(T^{-1}(z)) = T^{-1}(z)- \frac{g\left(T^{-1}(z)\right)}{g'\left(T^{-1}(z)\right)}\left(A+\frac{2(A-1)^2}{2(1-A)-L_{g}\left(T^{-1}(z)\right)} \right)$$
	$$with\quad L_{g}\left(T^{-1}(z)\right)= \frac{g\left(T^{-1}(z)\right)g''\left(T^{-1}(z)\right)}{\left(g'\left(T^{-1}(z)\right)\right)^2} \nonumber$$
	
	Since $\alpha  T^{-1}(z)+\beta=z$,  $g\circ T^{-1}(z)=f(z)$ and $\left(g\circ T^{-1}\right)'(z)=\frac{1}{\alpha} g'\left(T^{-1}(z)\right)$, we get $g'\left(T^{-1}(z)\right)=\alpha\left(g\circ T^{-1}\right)'(z)=\alpha f'(z)$, $g''\left(T^{-1}(z)\right)=\alpha^2 f''(z)$. We therefore have
	
	%\begin{equation*}
	%	T \circ R_g\circ T^{-1}(z) = T\left(R_g(T^{-1}(z))\right)=\alpha R_g\left(T^{-1}(z)\right)+\beta
	%\end{equation*}
	$T \circ R_g\circ T^{-1}(z) = T\left(R_g(T^{-1}(z))\right)=\alpha R_g\left(T^{-1}(z)\right)+\beta$
	\begin{eqnarray*}
		=&\alpha&  T^{-1}(z)- \frac{\alpha g\left(T^{-1}(z)\right)}{g'\left(T^{-1}(z)\right)}\left(A+\frac{2(A-1)^2}{2(1-A)-L_{g}\left(\frac{g\left(T^{-1}(z)\right)g''\left(T^{-1}(z)\right)}{\left(g'\left(T^{-1}(z)\right)\right)^2}\right)}  \right) +\beta\\
		=&z&-\frac{f(z)}{f'(z)}\left(A+\frac{2(A-1)^2}{2(1-A)-L_f(z)}\right)=R_f(z)
	\end{eqnarray*}	
	This completing the proof
\end{proof}

The  above scaling theorem indicates that up to a suitable change of coordinates the study of the dynamics of  the iteration function  (\ref{family}) for polynomials can be reduced to the study  of the dynamics of the same iteration functions for simpler polynomials.

\begin{defn} \cite{KKneisl2001}. We say that a one-point iterative root-finding algorithm $p\rightarrow T_p$ has a universal Julia set (for polynomials of degree d) if there exists a rational map $S$ such that for every degree $d$ polynomial $p$, $J(T_p)$ is conjugate by a M\"{o}bius transformation to $J(S)$
\end{defn}

The following theorem establishes a universal Julia set for quadratics for our methods (\ref{family}), (\ref{familym}), (\ref{family2}) and (\ref{GanderJarrattClass1}).

\begin{thm}\label{t:mio} For a rational map  $R_p(z)$ arising from the methods (\ref{family}),  (\ref{family2}) and (\ref{GanderJarrattClass1}) applied to $p(z)=(z-a)(z-b), a\neq b$,  $R_p(z)$  is conjugate via the M\"{o}bius transformation given by $M(z)=\frac{z-a}{z-b}$ to
	
	\begin{equation}\label{S}
		S(z)=\frac{z^3((A-1)z+2A-1)}{(2A-1)z+A-1}
	\end{equation}
	
\end{thm}

\begin{proof}
	Let $p(z)=(z-a)(z-b), a\neq b$ and let $M$ be the M\"{o}bius transformation  given by $M(z)=\frac{z-a}{z-b}$ with its inverse  given by $M^{-1}(u)=\frac{bu-a}{u-1}$, which may be considered as a map from $\mathbb{C}\cup\{\infty\}$. We then have
	
	$$M\circ R_p\circ M^{-1}(u)=M\circ R_p\left(\frac{bu-a}{u-1}\right)=\frac{u^3((A-1)u+2A-1)}{(2A-1)u+A-1}$$
	
\end{proof}

We observe that parameters $a$ and $b$ have been obviated in $S(z)$, as an effect of the Scaling Theorem that is verified by this family.

An identical result can be established for the iteration equation (\ref{familym}) when applied to the polynomial $p_m(z)=(z-a)^m(z-b)^m, a\neq b$

%%% ----------------------------------------------------------------------

\subsection{\textbf{Newton's Method $(A=1)$}}\label{ssec:Newton}

If $A=1$ then $S(z)=z^2$, fixed points $z=0$, $z=\infty$ and $z=1$. $S'=2z$ then  $|S'(1)|=2$. strange fixed point repulsive

\subsection{\textbf{Halley's Method $\left(A=0\right)$}}\label{ssec:Halley}

If $A=0$ then $S(z)=z^3$, fixed points $z=0$, $z=\infty$ and $z=\pm 1$.  $S'=3z^2$ then  $|S'(\pm 1)|=3$. strange fixed points repulsive

\subsection{\textbf{An interesting case $\left(A=\frac{2}{3}\right)$}}\label{ssec:int}

If $A=\frac{2}{3}$ then $S(z)=-z^3$, fixed points $z=0$ , $z=\infty$ and $z=\pm i$. $S'=-3z^2$ then  $|S'(\pm i)|=3$. strange fixed points repulsive

\subsection{\textbf{Super Halley method $\left(A=\frac{1}{2}\right)$}}\label{ssec:S-Halley}

If $A=\frac{1}{2}$ then $S(z)=z^4$, fixed points $z=0$ , $z=\infty$ and $z_k=e^{i2k\pi /3}, k=0,1,2$. $S'=4z^3$ then  $|S'(z_k)|=4, k=0,1,2$. strange fixed points repulsive

\subsection{\textbf{Study of the fixed points}}

The fixed points of $S(z)$ are the roots of the equation $S(z)=z$, that is, $z=0$,   $z=\infty$, $z=1$ and
\begin{equation}\label{}
	z_{2,3}=\frac{3A-2\pm\sqrt{A(5A-4)}}{2(1-A)}
\end{equation}

Observe that $z_2=\frac{1}{z_3}$ and $z_2=z_3$ only when $A$ is $0$ and $\frac{4}{5}$; in these cases $z_2=z_3=-1$ and  $z_2=z_3=1$ respectively.

In Figure \ref{Figure_Fixed_Points},  the behavior of the fixed points  for real values of $A$ between $-5$ and $5$ is represented.
\begin{figure}
	\centering
	\includegraphics[scale=1]{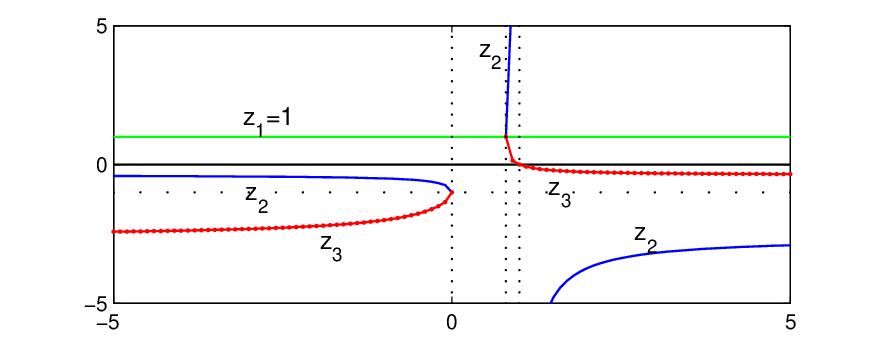}\\
	\caption{Dynamical Behavior of strange fixed points  for $-5 < A < 5$}\label{Figure_Fixed_Points}
\end{figure}

In order to study the stability of the fixed points, we calculate the first derivative of $S(z)$,

\begin{equation}\label{S'General}
	S'(z)=\frac{z^2\left[3(2A^2-3A+1)z^2+2(6A^2-8A+3)z+3(2A^2-3A+1)\right]}
	{\left[(2A-1)z+A-1\right]^2}
\end{equation}

Is obvious from (\ref{S'General}) that $z=0$ and $z=\infty$ are superatractive fixed points. The stability of the others fixed points changes depending on the values of the parameters $A$.

The operator $S'(z)$ in $z=1$ gives

\begin{equation}\label{S'1}
	|S'(1)|=2\left|\frac{4A-3}{3A-2}\right|,\quad \left(A\neq\frac{2}{3}\right)
\end{equation}

If we analyze this function, we obtain an horizontal asymptote in $|S'(1)|=\frac{4}{3}$ when $A\rightarrow \pm\infty$, and a vertical asymptote in $A=\frac{2}{3}$.

In the following result we present the stability of the  fixed point $z=1$

\begin{thm}\label{t:stability_z=1}
	The strange fixed point $z=1$ satisfies the following statements:
	
	\begin{enumerate}
		\item If $\left|A-\frac{42}{55}\right|<\frac{2}{55}$, then $z=1$ is an attractor and, in particular, it is a superattractor for $A=\frac{3}{4}$
		\item If $\left|A-\frac{42}{55}\right|=\frac{2}{55}$, then $z=1$ is a parabolic fixed point
		\item If $A\neq \frac{2}{3}$ and $\left|A-\frac{42}{55}\right|>\frac{2}{55}$, then $z=1$ is a repulsive fixed point.
	\end{enumerate}
\end{thm}

\begin{proof}
	From (\ref{S'1}),
	
	\begin{equation}\label{}
		|S'(1)|=2\left|\frac{4A-3}{3A-2}\right|\leq1\Rightarrow 2\left|4A-3\right|\leq\left|3A-2\right|
	\end{equation}
	
	Let $A=\alpha+i\beta$ be an arbitrary complex number. Then,
	
	\begin{equation}
		|4A-3|^{2}=(4\alpha-3)^{2}+(4\beta)^{2}\nonumber
	\end{equation}
	and
	\begin{equation}
		|3A-2|^{2}=(3\alpha-2)^{2}+(3\beta)^{2}\nonumber
	\end{equation}
	So
	
	\begin{equation}\label{}
		\left(\alpha-\frac{42}{55}\right)^2+\beta^2\leq\frac{4}{3025}\nonumber
	\end{equation}
	
	Therefore,
	\begin{equation}\label{}
		|S'(1)|\leq 1 \Leftrightarrow \left|A-\frac{42}{55}\right|\leq\frac{2}{55}\nonumber
	\end{equation}
	
	Finally, if $\left|A-\frac{42}{55}\right|>\frac{2}{55}$, then $|S'(1)|>1$ and $z=1$ is a repulsive fixed point.
\end{proof}

The operator $S'(z)$ in $z=z_{2,3}=\frac{3A-2\pm\sqrt{A(5A-4)}}{2(1-A)} (A\neq 0, A\neq 1)$ gives

\begin{equation}\label{}
	\left|S'\left(z_{2,3}\right)\right|=\frac{1}{2}\left|\frac{(6A-5)A\left(-2+3A\pm\sqrt{A(5A-4)}\right)^3}{\left(4A^2-3A\pm(2A-1)\sqrt{A(5A-4)}\right)^2(A-1)}\right|
\end{equation}

As $$\left(3A-2\pm\sqrt{A(5A-4)}\right)^3A=2\left(4A^2-3A\pm(2A-1)\sqrt{A(5A-4)}\right)^2$$ then

\begin{equation}\label{S'1}
	\left|S'\left(z_{2,3}\right)\right|=\left|\frac{6A-5}{A-1}\right|:\quad (A\neq 0, A\neq 1)
\end{equation}

If we analyze this function, we obtain an horizontal asymptote in $\left|S'\left(z_{2,3}\right)\right|=6$ when $A\rightarrow \pm\infty$, and  vertical asymptotes in $A=1$.

In the following result we present the stability of fixed points $z=z_{2,3}$

\begin{thm}\label{t:stability_z=xx}
	The strange fixed points $z_{2,3}=\frac{3A-2\pm\sqrt{A(5A-4)}}{2(1-A)}$ satisfies the following statements:
	
	\begin{enumerate}
		\item If $\left|A-\frac{29}{35}\right|<\frac{1}{35}$, then $z=z_{2,3}$ are attractors and, in particular, they are superattractors for $A=\frac{5}{6}$
		\item If $\left|A-\frac{29}{35}\right|=\frac{1}{35}$, then $z=z_{2,3}$ are parabolic fixed points
		\item If $A\neq 0$, $A\neq 1$ and  $\left|A-\frac{29}{35}\right|>\frac{1}{35}$, then $z=z_{2,3}$ are repulsive fixed points.
	\end{enumerate}
\end{thm}

\begin{proof}
	From (\ref{S'1}),
	
	\begin{equation}\label{S'1}
		\left|S'\left(z_{2,3}\right)\right|=\left|\frac{6A-5}{A-1}\right| \leq1\Rightarrow \left|6A-5\right|\leq\left|A-1\right|
	\end{equation}
	
	Let $A=\alpha+i\beta$ be an arbitrary complex number. Then,
	
	\begin{equation}
		|6A-5|^{2}=(6\alpha-5)^{2}+(6\beta)^{2}\nonumber
	\end{equation}
	and
	\begin{equation}
		|A-1|^{2}=(\alpha-1)^{2}+(\beta)^{2}\nonumber
	\end{equation}
	So
	
	\begin{equation}\label{}
		\left(\alpha-\frac{29}{35}\right)^2+\beta^2\leq\frac{1}{1225}\nonumber
	\end{equation}
	
	Therefore,
	\begin{equation}\label{}
		\left|S'\left(z_{2,3}\right)\right|=\left|\frac{6A-5}{A-1}\right| \leq1 \Leftrightarrow \left|A-\frac{29}{35}\right|\leq\frac{1}{35}\nonumber
	\end{equation}
	
	Finally, if $\left|A-\frac{29}{35}\right|>\frac{1}{35}$, then $|S'(z_{2,3})|>1$ and thus $z_{2,3}$ are repulsive fixed points.
\end{proof}

In  Figure \ref{Figure_Stability}  the functions where are observed the regions of stability are graphed. These functions are given by
\begin{eqnarray}
	S_1(1) &=& min\left\{|S'(1)|,1\right\} \\
	S_1(z_{2,3}) &=& min\left\{|S'(z_{2,3})|,1\right\}
\end{eqnarray}
So, zones of stability are when $S_1(A)=1$.

\begin{figure}
	\centering
	\includegraphics[scale=.8]{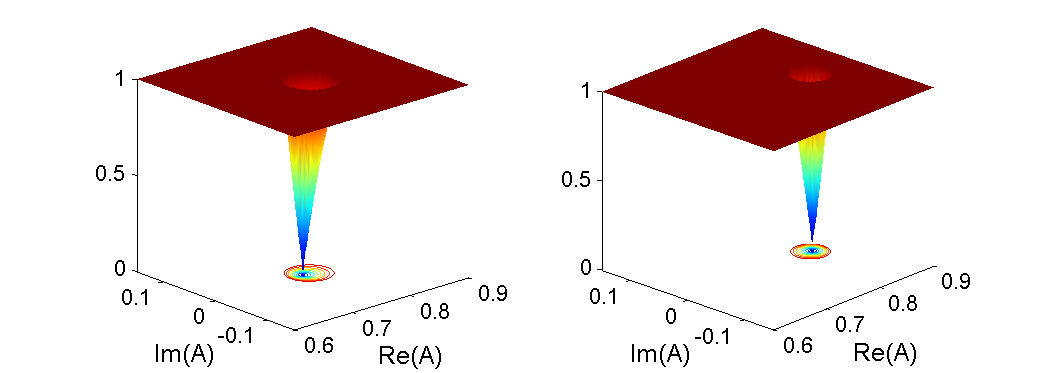}\\
	\caption{Stability regions of the strange fixed point. Left: $z=1$. Right: $z=z_{2,3}$}\label{Figure_Stability}
\end{figure}

%%% ----------------------------------------------------------------------

\subsection{\textbf{Study of the critical points}}

Critical points of $S(z)$ satisfy $S'(z)=0$, that is, $z=0$, $z=\infty$ and

\begin{equation}\label{Critical_Point_zc1}
	zc_1=\frac{6A^2-8A+3+\sqrt{12A^3-17A^2+6A}}{3(3A-1-2A^2)}
\end{equation}

\begin{equation}\label{Critical_Point_zc2}
	zc_2=\frac{6A^2-8A+3-\sqrt{12A^3-17A^2+6A}}{3(3A-1-2A^2)}
\end{equation}
Observe that $zc_1=\frac{1}{zc_2}$ and $zc_1=zc_2$ only when $A$ is $0,\frac{2}{3}$ and $\frac{3}{4}$; in these cases  $zc_1=zc_2=-1$ for $A=0$ and  $zc_1=zc_2=1$ for $A=\frac{2}{3}$ and $A=\frac{3}{4}$.

In Figure \ref{Figure_Critical_Points},  the behavior of the  critical points for real values of $A$ between $-1$ and $5$ is represented.
\begin{figure}
	\centering
	\includegraphics[scale=1]{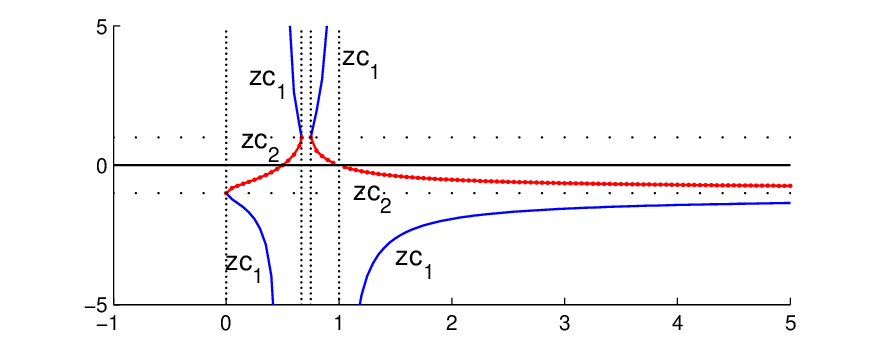}\\
	\caption{Dynamical Behavior of critical points for $-1 < A < 5$}\label{Figure_Critical_Points}
\end{figure}

\subsection{\textbf{Study of parameter space}}
In order to find the best members of the family in terms of stability, the parameter spaces associated to the free critical points $zc_1$ will be shown. It is well known that there is at least one critical point associated with each invariant Fatou component.  The parameter plane is obtained by iterating the selecting critical point; each point of the parameter plane is associated with a complex value of $A$, i.e., with an element of the family. Here, we using a mesh of $1000 \times 1000$ points, a maximum of $50$ iterations and a tolerance of $10^{-2}$ . Red color in Figures \ref{Figure_Espacio_Parametros_Puntos_Criticos2} and \ref{Figure_Espacio_Parametros_Puntos_Criticos3} means that the critical point is in the basin of attraction of $z=0$ or $z=\infty$, whereas that blue color indicates that the critical point generates iterations do not converge; another colors indicates convergence to strange points fixed. The ideas of \cite{CECadenas2017Cuencas} were used to create the figures.

\begin{figure}[h]
	\centering
	\includegraphics[scale=.8]{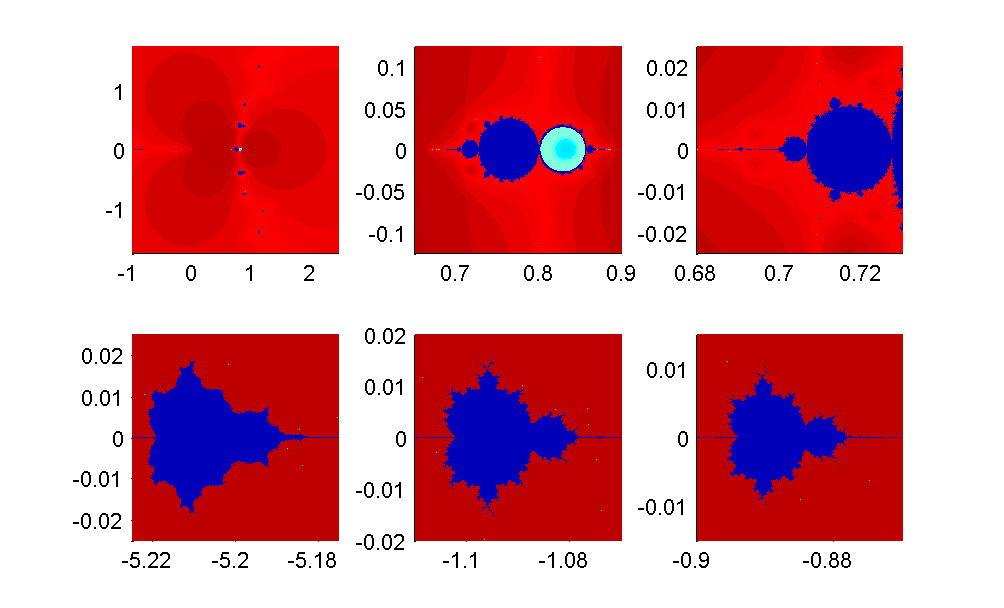}\\
	\caption{Parameter plane associated to the critical point $zc1$ and zoom}\label{Figure_Espacio_Parametros_Puntos_Criticos2}
\end{figure}

\begin{figure}[h]
	\centering
	\includegraphics[scale=.65]{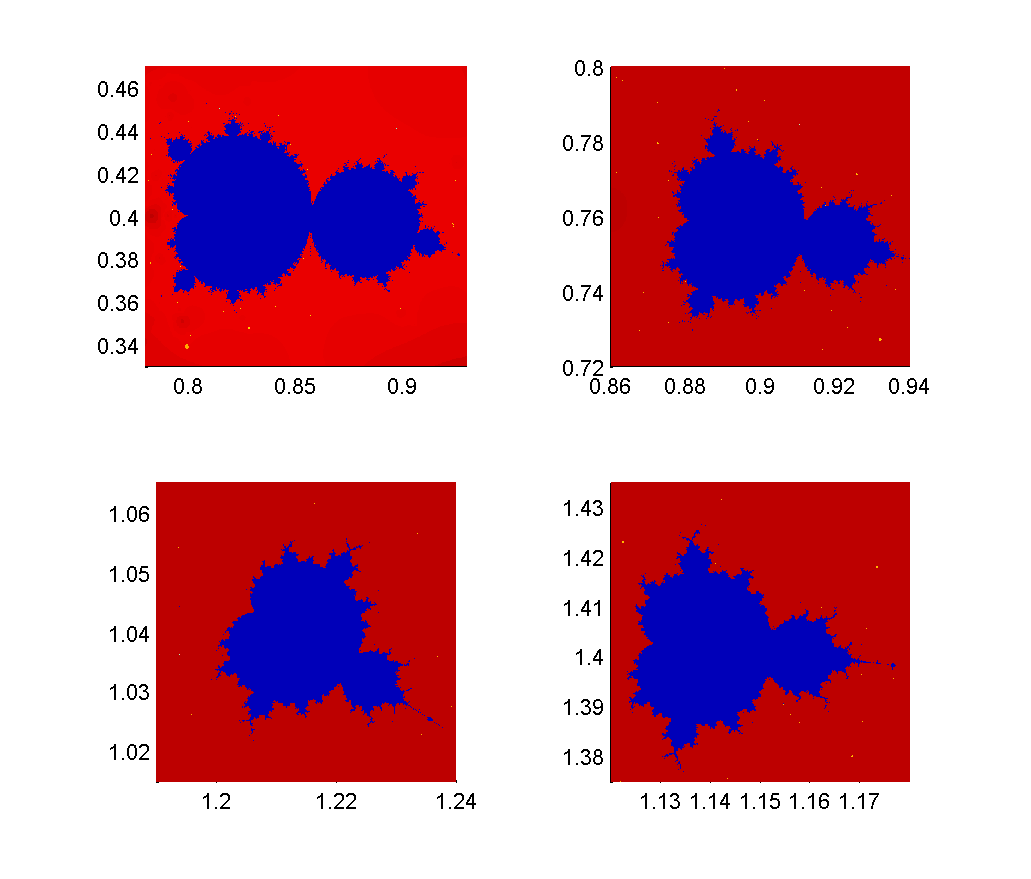}\\
	\caption{Various areas of the $zc1$-associated parameter space.}\label{Figure_Espacio_Parametros_Puntos_Criticos3}
\end{figure}

%\begin{figure}
%\centering
%  % Requires \usepackage{graphicx}
%  \includegraphics[scale=1]{Espacio_Parametros_Puntos_CriticosB.eps}\\
%  \caption{Parameter plane associated to the critical point $zc1$ and zoom}\label{Figure_ParameterB}
%\end{figure}
%
%

%%%%%%%%%%%%%%%%%%%%%%%%%%%%%%%%%%%%%%%%%%%%%%%%%%%%%%%%%%%%%%%%%%%%%%%%%%%%%%%%%%%%%%%%%%%%%%%%%%%%%%%%%%%%%%%%%%%%%%%%%%%%%%%%%%%%

\subsection{\textbf{Dynamical Planes}}

Then, focussing the attention in the regions shown in Figure \ref{Figure_Espacio_Parametros_Puntos_Criticos2} it is evident that
there exist members of the family with complicated behavior. In Figures \ref{Figure_Stables}-\ref{Figure_Dynamic1i} diverse stable dynamical planes are shown. In these dynamical planes the convergence to 0 appear in light blue, in red it appears the convergence
to $\infty$.

The dynamic planes that appear in Figure \ref{Figure_Stables} correspond to the methods given in subsections \ref{ssec:Newton}, \ref{ssec:Halley}, \ref{ssec:int} and \ref{ssec:S-Halley}. In all of them it is possible to observe the two well differentiated attraction basins delimited by a circumference of radius one.

In Figure \ref{Figure_Stables3} we can see several disconnected zones that correspond to the two attraction basins. We can also observe their behavior or evolution as the value of the parameter $A$ is increased.

In Figure \ref{Figure_Dynamic1} the dynamic planes are represented for three cases of negative values of the parameter $A$. Then, in Figure \ref{Figure_Dynamic1i} the value of the parameter is increased  in $i=\sqrt{-1}$, observing the distortion of the attraction basins.

%, in dark blue the zones with no convergence to the roots and others colors appears the convergence to strange fixed points.

\begin{figure}
	\centering
	\includegraphics[scale=.8]{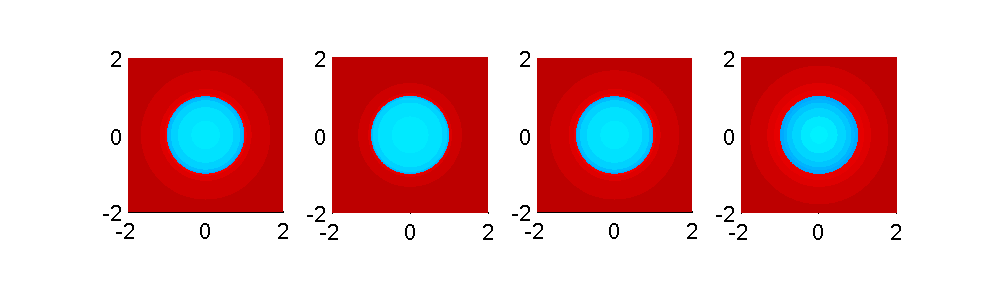}\\
	\caption{Dynamical planes. $A=0,\frac{1}{2},\frac{2}{3},1$.}\label{Figure_Stables}
\end{figure}

%\begin{figure}
%	\centering
%	% Requires \usepackage{graphicx}
%	\includegraphics[scale=.8]{Stables2.eps}\\
%	\caption{Dynamical planes. $A=-10,-5,-1,-\frac{1}{4}$.}\label{Figure_Stables2}
%\end{figure}

%\begin{figure}
%	\centering
%	% Requires \usepackage{graphicx}
%	\includegraphics[scale=.8]{Stables4.eps}\\
%	\caption{Dynamical planes. $A=-10+i,-5+i,-1+i,-\frac{1}{4}+i$.}\label{Figure_Stables4}
%\end{figure}

\begin{figure}
	\centering
	\includegraphics[scale=.8]{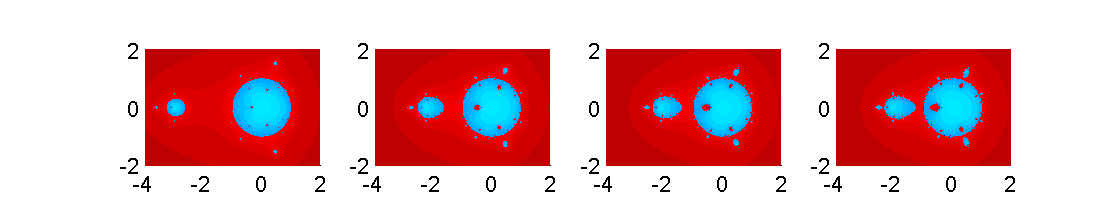}\\
	\caption{Dynamical planes. $A=2,5,10,20$.}\label{Figure_Stables3}
\end{figure}

%%% ----------------------------------------------------------------------

\begin{figure}
	\centering
	\includegraphics[scale=.8]{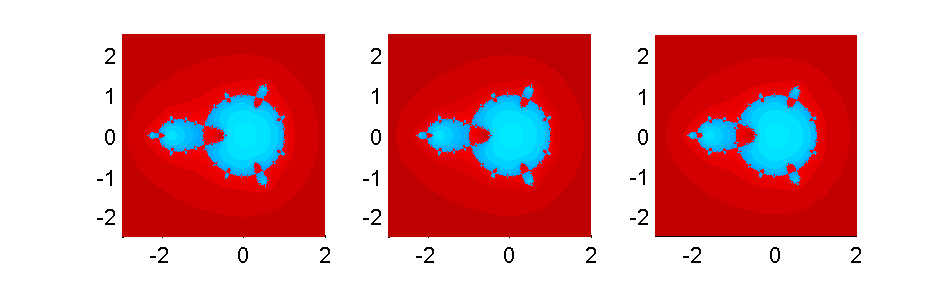}\\
	\caption{Dynamical planes. Left: $A=-4$. Center: $A=-3$. Right: $A=-2$}\label{Figure_Dynamic1}
\end{figure}

\begin{figure}
	\centering
	\includegraphics[scale=.8]{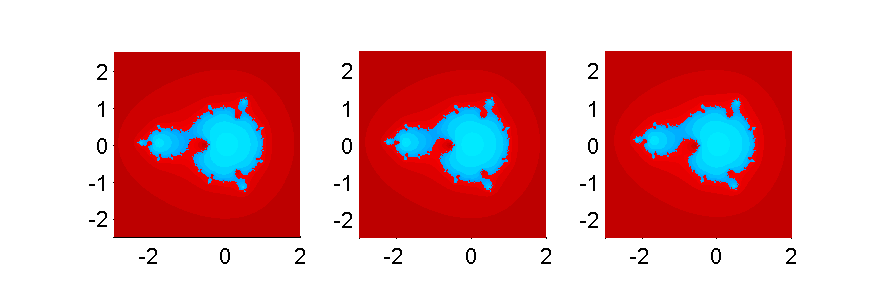}\\
	\caption{Dynamical planes. Left: $A=-4+i$. Center: $A=-3+i$. Right: $A=-2+i$}\label{Figure_Dynamic1i}
\end{figure}

%\begin{figure}
%\centering
%  % Requires \usepackage{graphicx}
%  \includegraphics[scale=.8]{Dynamic2.eps}\\
%  \caption{Dynamical planes. Left: $A=-1$. Center: $A=-\frac{3}{4}$. Right: $A=-\frac{1}{2}$.}\label{Figure_Dynamic2}
%\end{figure}

%\begin{figure}
%\centering
%  % Requires \usepackage{graphicx}
%  \includegraphics[scale=1]{Dynamic3.eps}\\
%  \caption{Dynamical planes. Left: $A=-\frac{1}{4}$. Center: $A=0$. Right: $A=\frac{1}{4}$.}\label{Figure_Dynamic3}
%\end{figure}

In Figure  \ref{Figure_InStablesPeriod2} left, a dynamical plane with  of a member of the family $Re(A)=\frac{3}{4}$ is shown. In this case only five periodic orbits of periodic two are shown. In Table \ref{TablePeriod2} points of the six periodic orbits of periodic two are presented.

%
%In Figures  \ref{Figure_InStablesPeriod2} right, the dynamical plane of a member of the family ($A=0.825$) with regions of convergence to any of the strange fixed points is shown.

In Figure \ref{Figure_InStablesPeriod2} to the right, the dynamic plane of a family member ($A=0.825$) is shown with regions of convergence to two of the strange fixed points (yellow and green).

\begin{figure}
	\centering
	\includegraphics[scale=.75]{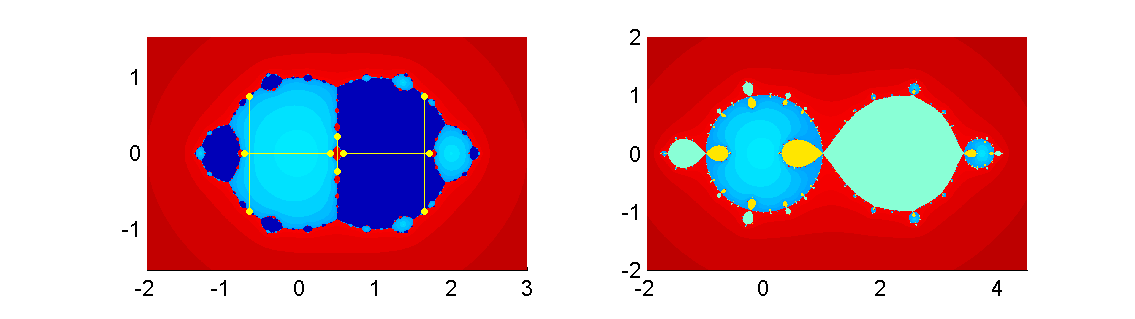}\\
	\caption{Dynamical planes. Left: $A=\frac{3}{4}$.  Right: $A=0.825$.}\label{Figure_InStablesPeriod2}
\end{figure}

\begin{table}\label{TablePeriod2}
	\begin{center}
		\begin{tabular}{|c|c|c|}
			\hline
			No & $z_1$ & $z_2$ \\\hline\hline
			1 & -0.7220838057 & 0.4194081680 \\\hline
			2 & 0.5806918199 & 1.722083806 \\\hline
			3 & -0.6513878189+0.7587449568 i & -0.6513878189-0.7587449568 i \\\hline
			4 & 0.5+0.2297294882 i & 0.5-0.2297294882 i \\\hline
			5 & 1.6513878189+0.7587449568 i & 1.6513878189-0.7587449568 i \\\hline
			6 & -1.384880802 & 2.384880802 \\
			\hline
		\end{tabular}
	\end{center}
	\caption{Periodic  orbits of period two.}
\end{table}

%\begin{figure}
%\centering
%  % Requires \usepackage{graphicx}
%  \includegraphics[scale=1]{Dynamic5.eps}\\
%  \caption{Dynamical planes.  $A=0.825$. }\label{Figure_Dynamic5}
%\end{figure}
%
%
%\begin{figure}
%\centering
%  % Requires \usepackage{graphicx}
%  \includegraphics[scale=1]{Dynamic6.eps}\\
%  \caption{Dynamical planes with complex parameters. Left: $A=-\frac{1}{4}+i$. Center: $A=\frac{1}{4}+i$ . Right: $A=\frac{1}{2}+i$.}\label{Figure_Dynamic6}
%\end{figure}
%
%
%\begin{figure}
%\centering
%  % Requires \usepackage{graphicx}
%  \includegraphics[scale=1]{Dynamic7.eps}\\
%  \caption{Dynamical planes with imaginary pure parameters. Left: $A=i$.  Right: $A=-i$.}\label{Figure_Dynamic7}
%\end{figure}

%
%
%\begin{figure}
%\centering
%  % Requires \usepackage{graphicx}
%  \includegraphics[scale=1]{Orbita_periodica.eps}\\
%  \caption{Dynamical planes with complex parameters . Left: 2-periodic orbit with $A=1.35$.  Right: zoom.}\label{Figure_Dynamic8}
%\end{figure}

% ------------------------------------------------------------------------

\section{Result and discussion}

In this paper, a complex dynamical study of  a retrieved family through a convex combination of the Newton's method and one Newton-Halley type method, on quadratic polynomial, is presented. Once the associated rational operator has been found, the fixed and critical points have been obtained and the stability of the strange fixed points are studied. Then, we identified areas both good and bad behavior in space of parameter associated with the free critical point. Of these areas the values of parameter $A$ were selected to illustrate the behavior of the method associated through the dynamical planes. Finally, six periodic orbits were calculated and five of them were represented in a figure.

\label{end-art}
\end{document}